\documentclass[10pt]{amsart}

\usepackage{latexsym}
\usepackage{amsmath}
\usepackage{amsfonts}
\usepackage{amssymb}
\usepackage{esint}
\usepackage{mathrsfs}

\usepackage{hyperref}

\newcommand{\spazio}{\ensuremath{\overline{H}\textsuperscript{\,$1$}}}

\begin{document}

\newtheorem{lem}{Lemma}[section]
\newtheorem{prop}[lem]{Proposition}
\newtheorem{thm}[lem]{Theorem}
\newtheorem{rem}[lem]{Remark}
\newtheorem{cor}[lem]{Corollary}
\newtheorem{defn}[lem]{Definition}

\title[The Mean Field Equation on Compact Surfaces] {An Existence Result for the Mean Field Equation on Compact Surfaces in a Doubly Supercritical Regime}

\author{Aleks Jevnikar}

\address{SISSA, via Bonomea 265, 34136 Trieste (Italy)}

\email{ajevnika@sissa.it}

\begin{abstract}
We consider a class of variational equations with exponential nonlinearities on a compact Riemannian surface, describing the mean field equation of the equilibrium turbulance with arbitrarily signed vortices. For the first time, we consider the problem with both supercritical parameters and we give an existence result by using variational methods. In doing this, we present a new
Moser-Trudinger type inequality under suitable conditions on the center of mass and the scale of concentration of both $e^u$ and $e^{-u}$, where $u$ is the unknown function in the equation. 
\end{abstract}

\maketitle

\section{Introduction}

In this paper we consider the equation
\begin{equation}
-\Delta_g u = \rho_1 \left(\frac{h_1(x)\,e^u}{\int_\Sigma h_1(x)\,e^u \,dV_g} - \frac{1}{|\Sigma|}\right) - \rho_2 \left(\frac{h_2(x)\,e^{-u}}{\int_\Sigma h_2(x)\,e^{-u} \,dV_g} - \frac{1}{|\Sigma|} \right) \hspace{0.3cm} \mbox{on $\Sigma$}, \label{eq}
\end{equation}
where $\rho_1, \rho_2$ are two non-negative parameters, $h_1,h_2:\Sigma \rightarrow \mathbb{R}$ are two smooth positive functions and $\Sigma$ is a compact orientable surface without boundary with Riemannian metric $g$ and volume $|\Sigma|$. 

This equation arises in mathematical physics as a mean field equation of the equilibrium turbulance with arbitrarily signed vortices, and is obtained by Joyce and Montgomery \cite{joy-mont} and  by Pointin and Lundgren \cite{point-lund} from different statistical arguments. Later, many authors worked on this model, see for example \cite{cho,lio,new,ohtsuka-suzuki} and the references therein.

Equation \eqref{eq} has a variational structure and solutions can be found as critical points of the functional
\begin{eqnarray} \label{func}
I_{\rho_1,\rho_2}(u) &=& \frac{1}{2}\int_\Sigma |\nabla_g u|^2 \,dV_g - \rho_1 \log\int_\Sigma h_1(x)\,e^u \,dV_g  - \rho_2\log\int_\Sigma h_2(x)\,e^{-u} \,dV_g+  \nonumber\\
                   & & +\,\, \rho_1\int_\Sigma u \,dV_g - \rho_2\int_\Sigma u \,dV_g , \quad u\in H^1(\Sigma),
\end{eqnarray}
where we have normalized the volume $|\Sigma|$ of $\Sigma$ by $|\Sigma|=1.$ The structure of the functional $I_{\rho_1,\rho_2}$ strongly depends on the parameters $\rho_1,\rho_2$. A Moser-Trudinger type inequality relative to this functional was proved in \cite{ohtsuka-suzuki}, and one has that
\[ 
\log\int_\Sigma e^{u-\bar{u}} \,dV_g  + \log\int_\Sigma e^{-u+\bar{u}} \,dV_g \leq \frac{1}{16\pi}\int_\Sigma |\nabla_g u|^2 \,dV_g + C_\Sigma, 
\]
where $\bar{u}$ denotes the average of $u$. By the above inequality, if we consider the case $(\rho_1,\rho_2)\in (0,8\pi)\times(0,8\pi)$, the functional $I_{\rho_1,\rho_2}$ is bounded from below and coercive, hence solutions can be found as global minima.

\smallskip

The value $8\pi$, or more in general $8\pi\mathbb{N}$, are \emph{critical} and the existence problem becomes subtler due to a loss of compactness. Even in the case $\rho_2=0$, namely the Liouville-type problem
\begin{equation}  
-\Delta_g u = \rho\left(\frac{h(x)\,e^u}{\int_\Sigma {h(x)\,e^u \,dV_g}} - \frac{1}{|\Sigma|}  \right) \hspace{0.3cm} \mbox{on $\Sigma$,} 
\label{equa1}\end{equation}
the existence problem is a difficult one, see \cite{cl,djlw1,nt}. To solve equation \eqref{eq} (or equation \eqref{equa1}) in this critical case, one always needs geometry conditions, see \cite{djlw1,zhou1}. For example, for equation \eqref{eq} with $\rho_1=8\pi$ and $\rho_2\in(0,8\pi]$, in \cite{zhou1} the author gave an existence result under suitable conditions on the Gaussian curvature $K(x)$ of $\Sigma$, namely $K(x)$ should satisfy
\[
8\pi-\rho_2-2K(x)>0 \qquad \mbox{for $x\in\Sigma$}.
\]

If $\rho_i>8\pi$ for some $i=1, 2$, then $I_{\rho_1,\rho_2}$ is unbounded
from below and a minimization technique is no more possible. In general, one needs to apply variational methods to obtain existence of critical points
(generally of saddle type) for $I_{\rho_1,\rho_2}$.

\smallskip

The case with $\rho_2=0$ (for instance equation \eqref{equa1}) has been very much studied in the literature. Again the problem has a variational structure and the associated functional is given by
\[ 
I_\rho(u) = \frac{1}{2}\int_\Sigma |\nabla_g u|^2 \,dV_g + \rho\int_\Sigma u \,dV_g -\rho\log \int_\Sigma h(x)\,e^u \,dV_g. 
\] 
There are by now many results regarding existence, compactness of solutions, bubbling behavior, etc, see \cite{djlw,djadli,mal,tar}. In particular, we have existence of solutions for equation \eqref{equa1} for $\rho\in(8k\pi,8(k+1)\pi)$ with $k\geq 1$, see for example \cite{mal}. This existence result is based on a detailed study of the topology of large negative sublevels of the functional $I_\rho$. It is indeed possible to find a homotopy equivalence between these sublevels and the so called \emph{space of formal baricentres} $\Sigma_k$, namely the family of elements $\sum_{i=1}^k t_i\delta_{x_i}$ with $(x_i)_i\subset\Sigma$ and $\sum_{i=1}^k t_i=1$, $t_i\geq0$. Exploiting the fact that the set $\Sigma_k$ is non contractible, it is then possible to introduce a min-max scheme based on this set.

\smallskip

On the other hand, in the case when $\rho_2\neq0$ and $\rho_i>8\pi$ for some $i=1, 2$, there are very few results. Here we point out some of them. The first is given in \cite{jwyz} and concerns with the case $\rho_1\in(8\pi,16\pi)$ and $\rho_2<8\pi$. Via a blow up analysis the authors proved existence of solutions for equation \eqref{eq} on a smooth, bounded, non simply-connected domain $\Sigma$ in $\mathbb{R}^2$ with homogeneous Dirichlet boundary condition. Later, in \cite{zhou} the author generalized this result to any compact surface without boundary by using analogous variational methods as those employed in the study of the problem \eqref{equa1}. In a certain sense, one can describe the topology of negative sublevels of the functional $I_{\rho_1,\rho_2}$ from the behaviour of the function $e^u$.

The blow up behaviour of solutions of equation \eqref{eq} is not yet developed in full generality. However, as in the case for $\rho_2=0$, in \cite{jwyz} the authors exhibited a volume quantization. More precisely, they proved that the blow up values are multiples of $8\pi$ (see the proof of Theorem \ref{thm-comp} for the definition of the blow up value). About this problem, by using a local quantization proved in \cite{ohtsuka-suzuki}, in Section 2 we deduce a global one for the case when $\rho_1,\rho_2\in(8\pi,16\pi)$.

\smallskip

We then turn to the existence issue and via a min-max scheme we obtain a positive result without any geometry and topology conditions. Our main theorem is the following:

\begin{thm}\label{main}
Assume that $\rho_1,\rho_2\in(8\pi,16\pi)$. Then there exists a solution to equation \eqref{eq}.
\end{thm}

The method to prove this existence result relies on a min-max scheme introduced by Malchiodi and Ruiz in \cite{mal-ruiz} for the study of Toda systems. Such a scheme is based on study of the topological properties of the low sublevels of $I_{\rho_1,\rho_2}$.

We shall see that on low sublevels of $I_{\rho_1,\rho_2}$ at least one of the functions $e^u$ or $e^{-u}$ is very concentrated around some point of $\Sigma$. Moreover, both $e^u$ and $e^{-u}$ can concentrate at two points that could eventually coincide, but in this case the \emph{scale of concentration} must be different. Roughly speaking, if $e^u$ and $e^{-u}$ concentrate around the same point at the same rate, then $I_{\rho_1,\rho_2}$ is bounded from below. We next make this statement more formal.

First, following the argument in \cite{mal-ruiz}, we define a continuous \emph{rate of concentration} $\sigma=\sigma(f)$ of a positive function $f\in\Sigma$, normalized in $L^1$. Somehow the smaller is $\sigma$, the
higher is the rate of concentration of $f$. Moreover we define a continuous \emph{center of mass} $\beta=\beta(f)\in\Sigma$. This can be done when $\sigma\leq\delta$ for some fixed $\delta$, therefore we have a map $\psi:H^1(\Sigma) \rightarrow \overline{\Sigma}_\delta$,
\[
\psi(u)= \bigr(\beta(f_1), \sigma(f_1)\bigr), \quad \psi(-u)= \bigr(\beta(f_2), \sigma(f_2)\bigr),
\]
where we have set
\[
f_1=\frac{e^{u}}{\int_{\Sigma} e^{u}\, dV_g}\,, \quad f_2=\frac{e^{-u}}{\int_{\Sigma} e^{-u}\, dV_g}.
\]
Here $\overline{\Sigma}_{\delta}$ is the topological cone over $\Sigma$, where we make the identification to a point when $\sigma \geq \delta $ for some $\delta>0$ fixed, see \eqref{cone}.

The improvement of the Moser-Trudinger inequality discussed above is made rigorous in the following way: if $ \psi(f_1) = \psi(f_2)$, then $I_{\rho_1,\rho_2}(u)$ is bounded from below, see Proposition \ref{mt-improved}. The proof is based on local versions of the Moser-Trudinger inequality on small balls and on annuli with small internal radius. We point out that our improved inequality is scaling invariant, differently from those proved by Chen-Li and Zhou (see \cite{chen-li} and \cite{zhou}).  

Using this fact, for $L>0$ large we can introduce a continuous map:
\[ 
I_{\rho_1,\rho_2}^{-L} \quad \xrightarrow{(\psi,\psi)}
\quad X:=\bigr(\overline{\Sigma}_{\delta} \times \overline{\Sigma}_{\delta}\bigr) \setminus \overline{D},
\]
where $\overline{D}$ is the diagonal of $\overline{\Sigma}_{\delta} \times
\overline{\Sigma}_{\delta}$ and $I_{\rho_1,\rho_2}^{-L} = \bigr\{ u\in H^1(\Sigma) : I_{\rho_1,\rho_2}(u)<-L \bigr\}$. On the other hand, it is also possible to do the converse, namely to map (a retraction of) the set $X$ into appropriate subevels of $I_{\rho_1,\rho_2}$. In Section 4 we construct a family of new test functions parametrized on (a suitable subset of) $X$ on which $I_{\rho_1,\rho_2}$ attains arbitrarily low values, see Proposition \ref{low}. Letting
\[
X \quad \stackrel{\phi}{\longrightarrow} \quad I_{\rho_1,\rho_2}^{-L}
\]
the corresponding map, it turns out that the composition of these two maps is homotopic to the identity on $X$, see Proposition \ref{hom}.

Exploiting the fact that $X$ is non-contractible, we are able to introduce a min-max argument to find a critical point of $I_{\rho_1,\rho_2}$. In this framework, an essential point is to use the `monotonicity argument' introduced by Struwe in \cite{struwe} jointly with the compactness result of solutions proved in Section 2, since it is not known whether the Palais-Smale condition holds or not.

\newpage

\begin{center}
\bfseries Acknowledgements
\end{center}
The author would like to thank Prof. A. Malchiodi for support and fundamental discussions about the topics of this paper.

\section{Notations and preliminaries}

In this section we fix our notation and recall some useful preliminary facts. Throughout the paper, $\Sigma$ stands for a compact orientable surface without boundary with metric $g$. For simplicity, we normalize the volume $|\Sigma|$ of $\Sigma$ by $|\Sigma|=1$. We state in particular some variants and improvements of the Moser-Trudinger type inequality and some of their conseguences.

We write $d(x,y)$ to denote the distance between two points $x,y\in\Sigma$. In the same way, for any $p\in\Sigma$ and $\Omega,\Omega' \subseteq \Sigma$, we denote:
\[ d(p,\Omega)=\inf\bigr\{ d(p,x):x\in\Omega \bigr\}, \hspace{0.5cm} d(\Omega,\Omega')=\inf\bigr\{ d(x,y):x\in\Omega, y\in\Omega' \bigr\}.   \]
Moreover, the symbol $B_p(r)$ stands for the open metric ball of radius $r$ and center $p$, while $A_p(r,R)$ for the open annulus of radii $r$ and $R$, $r<R$. The complement of a set $\Omega$ in $\Sigma$ will be denoted by $\Omega^c$.

Recalling that we are assuming $|\Sigma|=1$, given a function $u\in L^1(\Sigma)$, we denote its average as
\[ \bar{u} = \int_\Sigma u\,dV_g. \]

Given $\delta>0$, we define the topological cone:
\begin{equation}  \overline{\Sigma}_{\delta} =  \bigr(\Sigma \times (0, +\infty)\bigr)\Bigr/ \bigr(\Sigma \times [\delta, + \infty)\bigr), \label{cone} \end{equation}
where the equivalence relation identifies $\Sigma \times [\delta, + \infty)$ to a single point.

Throughout the paper we will denote by $C$ large constants which
are allowed to vary among different formulas or even within lines.
When we want to stress the dependence of the constants on some
parameter (or parameters), we add subscripts to $C$, as $C_\delta$,
etc.. Also constants with subscripts are allowed to vary.
Moreover, sometimes we will write $o_{\alpha}(1)$ to denote
quantities that tend to $0$ as $\alpha \to 0$ or $\alpha \to
+\infty$, depending on the case. We will similarly use the symbol
$O_\alpha(1)$ for bounded quantities.

\bigskip

\noindent We begin with a compactness result which is deduced from the blow up theorem in \cite{ohtsuka-suzuki}.

\begin{thm}\label{thm-comp}
Suppose that $u_n$ satisfies
\[ -\Delta_g u_n = \rho_{1,n} \left(\frac{h_1(x)\,e^{u_n}}{\int_\Sigma h_1(x)\,e^{u_n} \,dV_g} - \frac{1}{|\Sigma|}\right) - \rho_{2,n} \left(\frac{h_2(x)\,e^{-u_n}}{\int_\Sigma h_2(x)\,e^{-u_n} \,dV_g} - \frac{1}{|\Sigma|} \right) \hspace{0.3cm} \mbox{on $\Sigma$}. \]
Assume that $\rho_{1,n},\rho_{2,n}\in(8\pi,16\pi)$ for any $n\in\mathbb{N}$ and that $\rho_{1,n}\rightarrow \rho_1\in(8\pi,16\pi)$ and $\rho_{2,n}\rightarrow \rho_2\in(8\pi,16\pi)$. Then the solution sequence $(u_n)_n$ (up to adding suitable constants) is uniformly bounded in $L^{\infty}(\Sigma)$ and there exist $u$ and a subsequence $(u_{n_k})_k$ such that
\[ u_{n_k} \rightarrow u, \]
where this $u$ is a solution to (\ref{eq}) for these $\rho_1$ and $\rho_2$.
\end{thm}

\begin{proof}
Since $I_{\rho_1,\rho_2}$ is invariant under translation by constants in the argument, we can restrict ourselves to considering the subspace of $H^1(\Sigma)$ of functions with zero average.

Consider the blow up sets of the sequence $(u_n)_n$ given by
\[ S_1 = \Bigr\{ x\in\Sigma : \exists x_n\rightarrow x \mbox{\,\,such\,\,that\,\,} u_n(x_n)\rightarrow +\infty \Bigr\}, \]
\[ S_2 = \Bigr\{ x\in\Sigma : \exists x_n\rightarrow x \mbox{\,\,such\,\,that\,\,} u_n(x_n)\rightarrow -\infty \Bigr\}. \]
From the blow up theorem in \cite{ohtsuka-suzuki}, it is sufficient to show that $S_1 \cap S_2 = \emptyset$. We argue by contradiction. Assume that $x_0\in S_1 \cap S_2$. Define the blow up values at $x_0$ by
\[ m_1(x_0) = \lim_{r\rightarrow 0}\lim_{n\rightarrow +\infty} \int_{B_r(x_0)} \frac{\rho_{1,n}h_1(x)\,e^{u_n}}{\int_\Sigma h_1(x)\,e^{u_n}\,dV_g}\,dV_g, \]
\[ m_2(x_0) = \lim_{r\rightarrow 0}\lim_{n\rightarrow +\infty} \int_{B_r(x_0)} \frac{\rho_{2,n}h_2(x)\,e^{-u_n}}{\int_\Sigma h_2(x)\,e^{-u_n}\,dV_g}\,dV_g. \]
Since $\rho_{1,n},\rho_{2,n}\in(8\pi,16\pi)$, from the blow up theorem in \cite{ohtsuka-suzuki}, we have
\begin{equation}\label{bound}
4\pi\leq m_1(x_0)<16\pi, \hspace{0.5cm} 4\pi\leq m_2(x_0)<16\pi, 
\end{equation}
and
\begin{equation}\label{propr} 
\bigr(m_1(x_0)-m_2(x_0)\bigr)^2 = 8\pi\bigr(m_1(x_0)+m_2(x_0)\bigr). 
\end{equation}
By the last equality we derive
\[ 
m_1(x_0) = m_2(x_0)+4\pi \pm 4\sqrt{\pi m_2(x_0)+\pi^2}. 
\]
First, let us consider the case $m_1(x_0) = m_2(x_0)+4\pi + 4\sqrt{\pi m_2(x_0)+\pi^2}.$ Using the fact that $4\pi\leq m_2(x_0)$, we derive that $m_1(x_0)\geq 16\pi$, which is a contradiction to the first estimate in \eqref{bound}. 

If instead we consider the case $m_1(x_0) = m_2(x_0)+4\pi - 4\sqrt{\pi m_2(x_0)+\pi^2}$, the estimate $4\pi\leq m_2(x_0)<16\pi$ implies that $m_1(x_0)<12\pi$. By interchanging the roles of $m_1(x_0)$ and $m_2(x_0)$, we obtain the same inequality for $m_2(x_0)$. Therefore we have
\begin{equation}\label{bound2}
4\pi\leq m_1(x_0)<12\pi, \hspace{0.5cm} 4\pi\leq m_2(x_0)<12\pi. 
\end{equation} 
On the other hand, using \eqref{propr} jointly with the fact that $m_i(x_0)\geq 4\pi$, $i=1,2$, we deduce that
\[
|m_1(x_0)-m_2(x_0)|\geq 8\pi,
\]
which is a contradiction to \eqref{bound2}.
\end{proof}

Next, we recall some Moser-Trudinger type inequalities by starting with the standard one, i.e. for $u\in H^1(\Sigma)$ it holds
\begin{equation} \log \int_\Sigma e^{u-\bar{u}}\,dV_g \leq \frac{1}{16\pi}\int_\Sigma |\nabla_g u|^2\,dV_g  + C_\Sigma. \label{mt}\end{equation}
As observed in the introduction, problem (\ref{eq}) is the Euler-Lagrange equation of the functional $I_{\rho_1,\rho_2}$ given in (\ref{func}). If we consider the space 
\[ \spazio(\Sigma) = \left\{ u\in H^1(\Sigma): \int_\Sigma u \,dV_g=0 \right\}, \]
the following result has been proved by Ohtsuka and Suzuki in \cite{ohtsuka-suzuki}.

\begin{thm}
The functional $I_{\rho_1,\rho_2}$ is bounded from below on $\spazio(\Sigma)$ if and only if $\rho_i\leq8\pi$, $i=1,2$.
\end{thm}
 
In view of this result, similarly to inequality (\ref{mt}), we can also obtain a Moser-Trudinger inequality with $e^u$ and $e^{-u}$ simultaneously. Namely for $u\in H^1(\Sigma)$ it holds
\begin{equation} \log\int_\Sigma e^{u-\bar{u}} \,dV_g  + \log\int_\Sigma e^{-u+\bar{u}} \,dV_g \leq \frac{1}{16\pi}\int_\Sigma |\nabla_g u|^2 \,dV_g + C_\Sigma. \label{mt2}\end{equation}
It is well known that an improved inequality will hold if $e^u$ has integral bounded from below on different regions of $\Sigma$ of positive mutual distance.

\begin{prop}\emph{(\cite{zhou})}
For a fixed integer $l$, let $\Omega_1,\dots,\Omega_l$ be subsets of $\Sigma$ satisfying $d(\Omega_i,\Omega_j)\geq\delta_0$ for $i\neq j$, where $\delta_0$ is a positive real number, and let $\gamma_0\in\bigr(0,\frac{1}{l}\bigr)$. Then, for any $\varepsilon>0$ there exists a constant $C=C(\Sigma,l,\varepsilon,\delta_0,\gamma_0)$ such that
\[ l\log\int_\Sigma e^{u-\bar{u}}\,dV_g + \log\int_\Sigma e^{-u+\bar{u}}\,dV_g \leq \frac{1}{16\pi-\varepsilon}\int_\Sigma |\nabla_g u|^2\,dV_g  + C \]
for all the functions $u\in H^1(\Sigma)$ satisfying
\[ \frac{\int_{\Omega_i}e^u\,dV_g}{\int_\Sigma e^u\,dV_g} \geq \gamma_0, \hspace{0.5cm} \forall\,i\in \{1,\dots,l\}. \]
\end{prop} 

We next state a result which is a local version of the inequality (\ref{mt2}), that will be of use later on.

\begin{prop}\label{mt-local}
Fix $\delta>0$, and let $\Omega_1\subset\Omega_2\subset\Sigma$ be such that
$d(\Omega_1, \partial \Omega_2) \geq \delta$. Then, for any $\varepsilon > 0$
there exists a constant $C = C(\varepsilon, \delta)$ such that for all $u
\in H^1(\Sigma)$
\[ \log \int_{\Omega_1}  e^{u}\, dV_g + \log
\int_{\Omega_1} e^{-u}\, dV_g \leq \frac{1}{16\pi-\varepsilon} \int_{\Omega_2} |\nabla_g u|^2 \,dV_g + C. \]
\end{prop}

\begin{proof}
The proof is developed exactly as in Proposition 2.3 of \cite{mal-ruiz}, with obvious modifications. Here we just sketch the proof for the reader's convenience. First, we consider a spectral decomposition of the Laplacian on $\Omega_2$ (with Neumann boundary conditions), in order to write $u$ as $u=v+w$ with $v\in L^\infty(\Omega_2)$ and $w\in H^1(\Omega_2)$. We next consider a smooth cutoff function $\chi$ with values into $[0,1]$ satisfying
$$
  \left\{
    \begin{array}{ll}
      \chi(x) = 1 & \hbox{ for } x \in \Omega_1,\\
      \chi(x) = 0 & \hbox{ if } d(x, \Omega) > \delta/2,
    \end{array}
  \right.
$$
and then define $\tilde{w}(x) = \chi(x) w(x)$. We now apply the Moser-Trudinger inequality (\ref{mt2}) to $\tilde{w}$ to deduce the desired inequality.
\end{proof}

We give now a criterion which is a first step in studying the properties of the low sublevels of $I_{\rho_1,\rho_2}$. We first state a lemma concerning a covering argument, which is a particular case of a more general setting in \cite{mal-ruiz}, Lemma 2.5.

\begin{lem} \label{cover}
Let $\delta_0 > 0$, $\gamma_0 > 0$ be fixed, and let $\Omega_{i,j} \subseteq \Sigma$,
$i,j = 1, 2$, satisfy $d(\Omega_{i,j},\Omega_{i,k}) \geq \delta_0$
for $j \neq k$. Suppose that $u \in H^1(\Sigma)$ is a
function verifying
\[
    \frac{\int_{\Omega_{1,j}} e^{u} \,dV_g}{\int_\Sigma e^{u} \,dV_g}
    \geq \gamma_0, \hspace{0.3cm} \frac{\int_{\Omega_{2,j}} e^{-u} \,dV_g}{\int_\Sigma e^{-u} \,dV_g}
    \geq \gamma_0, \hspace{0.3cm} j=1,2.  
\]
Then there exist positive constants $\tilde{\gamma}_0$, $\tilde{\delta}_0$,
depending only on $\gamma_0$, $\delta_0$, and two sets $\tilde{\Omega}_1, \tilde{\Omega}_2\subseteq \Sigma$, depending also on $u$ such that
\[
    d(\tilde{\Omega}_1, \tilde{\Omega}_2) \geq \tilde{\delta}_0; \qquad
    \quad \frac{\int_{\tilde{\Omega}_{i}} e^{u} \,dV_g}{\int_\Sigma e^{u} \,dV_g}
    \geq \tilde{\gamma}_0, \quad \frac{\int_{\tilde{\Omega}_{i}} e^{-u} \,dV_g}{\int_\Sigma
    e^{-u} \,dV_g} \geq \tilde{\gamma}_0; \quad i = 1, 2.
\]
\end{lem}

Using this result it is indeed possible to obtain an improvement of the constant in the Moser-Trudinger inequality (\ref{mt2}). 

\begin{prop}\label{mt-imp}
Let $u\in H^1(\Sigma)$ be a function satisfying
the assumptions of Lemma \ref{cover} for some positive constants
$\delta_0, \gamma_0$. Then for any $\varepsilon > 0$ there exists $C=C(\varepsilon) > 0$,
depending on $\varepsilon, \delta_0$, and $\gamma_0$ such that
\[
  \log \int_{\Sigma} e^{u-\bar{u}} \,dV_g + \log
  \int_{\Sigma} e^{-u+\bar{u}} \,dV_g \leq \frac{1}{32\pi-\varepsilon} \int_\Sigma       |\nabla_g u|^2 \,dV_g + C.
\]
\end{prop}

\begin{proof}
To obtain the thesis we can argue exactly as in Proposition 2.6 of \cite{mal-ruiz}. First we set $\tilde{\delta}_0, \tilde{\gamma}_0$ and $\tilde{\Omega}_1, \tilde{\Omega}_2$ as
in Lemma \ref{cover}. Then we apply Proposition \ref{mt-local} with $\tilde{\Omega}_i$ and $U_i=\bigr\{x \in \Omega:\
d(x, \tilde{\Omega}_i) < \tilde{\delta}_0/2 \bigr\}$ for $i=1,2$. Observing that
\[ \log \int_{\tilde{\Omega}_i}  e^{u} dV_g\geq \log \left ( \int_{\Sigma}  e^{u}
dV_g\right ) + \log \tilde{\gamma}_0,
\]
\[ \log \int_{\tilde{\Omega}_i}  e^{-u} dV_g\geq \log \left ( \int_{\Sigma}  e^{-u}
dV_g\right ) + \log \tilde{\gamma}_0 \]
for $i=1,2,$ and that $U_1\cap U_2=\emptyset$, we deduce the thesis.
\end{proof}

Proposition \ref{mt-imp} implies that on low sublevels of the functional $I_{\rho_1,\rho_2}$, at least one of the components of the couple $(e^u,e^{-u})$ must be very concentrated around a certain point. We will present in the sequel a more detailed description of the topology of low sublevels.

\section{Improved inequality}

Following the ideas presented by Malchiodi and Ruiz in \cite{mal-ruiz}, in this section we exhibit an improved Moser-Trudinger inequality under suitable conditions of concentration of the involved function.

\smallskip

First, we give continuous definitions of \emph{center of mass} and \emph{scale
of concentration} of positive functions normalized in $L^1$. Let us consider the set
\[
  A = \left\{ f \in L^1(\Sigma) \; : \; f >
  0 \ \hbox{ a. e. and } \int_\Sigma f dV_g = 1 \right\},
\]
endowed with the topology inherited from $L^1(\Sigma)$. Then we have the following result.

\begin{prop}\emph{(\cite{mal-ruiz})}\label{conc}
Let us fix a constant $R>1$. Then there exist \,\mbox{$\delta=\delta(R)\!>\!0$} and a continuous map:
\[ \psi : A \rightarrow \overline{\Sigma}_{\delta}, \qquad  \psi(f)= (\beta, \sigma), \]
satisfying the following property: for any $f \in A$ there exists
$p \in \Sigma$ such that

\begin{enumerate}

\item[{\emph a)}] $ d(p, \beta) \leq  C' \sigma$ for $C' = \max\bigr\{3
R+1, \delta^{-1}diam(\Sigma) \bigr\}.$

\item[{\emph b)}] There holds: 
\[ \int_{B_p(\sigma)} f \, dV_g >
\tau, \qquad  \int_{B_p(R \sigma)^c} f \, dV_g > \tau, \]
where $\tau>0$ depends only on $R$ and $\Sigma$.

\end{enumerate}
\end{prop}

\noindent This result is obtained in several steps, which we summarize in the sequel. The explicit definition of the map $\psi(f)= (\beta, \sigma)$ is given below. 

First, take $R_0=3R$, and define $ \sigma: A \times \Sigma \rightarrow (0,+\infty)$ such that:
\begin{equation}\label{def-sigma} 
\int_{B_x(\sigma(x,f))} f \, dV_g = \int_{B_x(R_0\sigma(x,f))^c} f \, dV_g. 
\end{equation}
The map $\sigma(x,f)$ is clearly uniquely determined and continuous. Moreover we have the following lemma.

\begin{lem}\emph{(\cite{mal-ruiz})}
The map $\sigma$ satisfies:
\begin{equation} \label{dist} 
d(x,y) \leq R_0 \max \bigr\{ \sigma(x,f), \sigma(y,f)\} +\min \{ \sigma(x,f), \sigma(y,f) \bigr\}.
\end{equation}
\end{lem}

We now define
$$
  T: A \times \Sigma \rightarrow \mathbb{R}, \qquad T(x,f) = \int_{B_x(\sigma(x,f))} f \,dV_g.
$$

\begin{lem}\emph{(\cite{mal-ruiz})}\label{sigma} 
If $x_0 \in \Sigma$ is such that $T(x_0,f) = \max_{y \in\Sigma} T(y,f)$, then we have $\sigma(x_0,f) < 3\,\sigma(x, f)$ for any other $x \in \Sigma$.
\end{lem}

As a consequence of the previous lemma, one can obtain the following:

\begin{lem}\emph{(\cite{mal-ruiz})}\label{tau} 
There exists a fixed $\tau > 0$ such that
$$
  \max_{x \in \Sigma} T(x,f) > \tau > 0 \qquad \hbox{ for all } f \in A.
$$
\end{lem}

Let us define
$$ \sigma : A \rightarrow \mathbb{R}, \qquad  \sigma(f)= 3 \min\bigr\{ \sigma(x,f): \ x \in \Sigma \bigr\},$$ 
which is obviously a continuous function. Given $\tau$ as in Lemma \ref{tau}, consider the set
\begin{equation} \label{defS}
  S(f) = \bigr\{ x \in \Sigma \; : \; T(x,f) > \tau,\ \sigma(x,f) <  \sigma(f) \bigr\},
\end{equation}
which is a nonempty open set for any $f\in A$, by Lemmas \ref{sigma} and \ref{tau}. Moreover, from \eqref{dist}, we have that
\begin{equation}\label{diamS} 
diam\bigr(S(f)\bigr) \leq (R_0+1)\sigma(f). 
\end{equation}
By the Nash embedding theorem, we can assume that $\Sigma \subset
\mathbb{R}^N$ isometrically, $N \in \mathbb{N}$. Take an open tubular neighborhood $\Sigma \subset U \subset \mathbb{R}^N$ of $\Sigma$, and $\delta>0$ small enough so that
\begin{equation}\label{co}
 co \bigr [ B_x\bigr((R_0+1)\delta\bigr)\cap \Sigma \bigr ] \subset
 U \quad \forall \, x \in \Sigma, 
\end{equation}
where $co$ denotes the convex hull in $\mathbb{R}^N$.

We define now
\[
  \eta(f) = \frac{\displaystyle \int_\Sigma \bigr(T(x,f) - \tau\bigr)^+ \bigr( \sigma(f) - \sigma(x,f) \bigr)^+ x  \,dV_g}{\displaystyle \int_\Sigma \bigr(T(x,f) - \tau\bigr)^+ \bigr( \sigma(f) - \sigma(x,f) \bigr)^+ \,dV_g}\in \mathbb{R}^N.
\]
The map $\eta$ defines a sort of center of mass in $\mathbb{R}^N$. Observe
that the integrands become nonzero only on the set $S(f)$. Moreover, whenever $\sigma(f) \leq \delta$, \eqref{diamS} and \eqref{co} imply that $\eta(f) \in U$, and so we can define
\[
\beta: \bigr\{f \in A:\ \sigma(f)\leq \delta \bigr\} \rightarrow \Sigma, \ \ \beta(f)= P
\circ \eta (f),
\]
where $P: U \rightarrow \Sigma$ is the orthogonal projection. 

Then the map $\psi(f)=\bigr(\beta(f), \sigma(f)\bigr)$ satisfies the conditions given by Proposition \ref{conc}. If $\sigma(f) \geq \delta$, $\beta$ is not defined. Observe that {\it a)} is then satisfied for any $\beta \in \Sigma$.

\begin{rem}
The above map $\psi(f)= (\beta, \sigma)$ gives us a center of mass of $f$ and its scale of concentration around that point. The identification in $\overline{\Sigma}_\delta$ is somehow natural, indeed, if $\sigma$ exceeds a certain positive constant, we do not have concentration at a point and so $\beta$ could not be defined.
\end{rem}

\smallskip

We next state an improved Moser-Trudinger inequality for functions $u\in H^1(\Sigma)$ such that both $e^u$ and $e^{-u}$ are concentrated at the same point with the same rate of concentration. In terms of Proposition \ref{conc}, we have the following result.

\begin{prop} \label{mt-improved}
Given any $\varepsilon>0$, there exist $R=R(\varepsilon)>1$ and $\psi$ as given in Proposition \ref{conc}, such that for any $u \in H^1(\Sigma)$ with:

$$\psi \left( \frac{e^{u}}{\int_{\Sigma} e^{u} dV_g} \right
)= \psi \left( \frac{e^{-u}}{\int_{\Sigma} e^{-u} dV_g} \right
),
$$
the following inequality holds:
\[ \log \int_\Sigma e^{u-\bar{u}} \,dV_g +\log
\int_\Sigma e^{-u+\bar{u}} \,dV_g \leq \frac{1}{32\pi-\varepsilon}\int_\Sigma |\nabla_g u|^2 \,dV_g + C, \]
for some $C=C(\varepsilon)$.
\end{prop}

Before proving the proposition, we need some preliminary lemmas concerning Moser-Trudinger type inequality for small balls, and also for annuli with small internal radius. The first one is obtained just by using a dilation argument.

\begin{lem} \label{ball} 
For any $\varepsilon>0$ there exists $C=C(\varepsilon)>0$ such that
\[  \log \int_{B_p(s/2)} e^{u}\,dV_g + \log \int_{B_p(s/2)} e^{-u}\,dV_g  \,\leq\, \frac{1}{16\pi-\varepsilon}\int_{B_p(s)}|\nabla_g u|^2\,dV_g +4\log s+ C \]
for any $u \in H^1(\Sigma), \ p \in \Sigma$, $s>0$ small.
\end{lem}

\begin{proof}
Notice that, as $s \to 0$ we consider quantities defined on smaller and smaller geodesic balls $B_p(\xi)$ on $\Sigma$. By considering normal geodesic coordinates at $p$, gradients, averages and the volume element will almost correspond to the Euclidean ones. If we assume that near $p$ the metric of $\Sigma$ is flat, we will get negligible error terms which will be omitted.

\smallskip

We just perform a convenient dilation of $u$ given by
\[ v(x)= u(s x+ p). \]
We have the following equalities:
\[ \int_{B_p(s)}|\nabla_g u|^2 \,dV_g = \int_{B_0(1)}|\nabla_g v|^2 \,dV_g, \]
\[ \int_{B_p(s/2)}e^u \,dV_g = s^2 \int_{B_0(1/2)}e^v \,dV_g. \]
We apply then Proposition \ref{mt-local} to the function $v$ to deduce the desired inequality. 
\end{proof}

\begin{rem}
Observe that in Lemma \ref{ball} and in the results that will be present in the sequel there is no explicit dependence of the average of $u$, due to the fact that the average of $u$ is cancelled by the average of $-u$.
\end{rem}

We next deduce a Moser-Trudinger type inequality on thick annuli. In order to do this, we use the Kelvin transform to exploit the geometric properties of the problem.

\begin{lem} \label{mt-annulus} 
Given $\varepsilon>0$, there exists a fixed $r_0>0$ (depending only on $\Sigma$ and $\varepsilon$) satisfying the following property: for any $r\!\in \!(0,r_0)$ fixed, there exists \mbox{$C\!=C(r,\varepsilon)>\!0$} such that, for any $u\in H^1(\Sigma)$ with $u=c\in\mathbb{R}$ in $\partial B_p(2 r)$,
\[ 
\log \int_{A_p(s, r)} \!\!e^{u} \,dV_g + \log \int_{A_p(s,r)} \!\!e^{-u} \,dV_g \leq  \frac{1}{16\pi-\varepsilon}\int_{A_p(s/2,2 r)}\!\!|\nabla_g u|^2 \,dV_g - 4\log s + C, 
\]
with $\ p \in \Sigma$, $s\in(0,r)$.
\end{lem}

\begin{proof}
As in the proof of Lemma \ref{ball}, by taking $r_0$ small enough,
also here the metric becomes close to the Euclidean one. We can then assume that the metric is flat around $p$.

\smallskip

We consider the Kelvin transform $ K : A_p(s/2, 2r) \rightarrow A_p(s/2, 2r)$ given by
\[ K(x)= p+ r s \frac{x-p}{\ |x-p|^2}. \]
Observe that $K$ maps the interior boundary of $A_p(s/2, 2r)$ onto the exterior one and viceversa. We next define the function $\tilde{u}\in H^1(B_p(2r))$ as:
\[ 
\tilde{u}(x)= \left \{ \begin{array}{ll} u\bigr(K(x)\bigr)  & \mbox{ if } |x-p| \geq s/2, \\
c & \mbox{ if } |x-p| \leq s/2. \end{array} \right. 
\]
Our goal is to apply the local Moser-Trudinger inequality given by Proposition \ref{mt-local} to $\tilde{u}$. First of all, observe that
\begin{equation} \label{exp} 
\int_{A_p(s, r)} e^{\tilde{u}} \,dV_g =  \int_{A_p(s,r)} e^{u(K(x))} \,dV_g = \int_{A_p(s, r)} e^{u(x)} \frac{s^2 r^2}{|x-p|^4} \,dV_g, 
\end{equation} 
since the Jacobian of $K$ is $J\bigr(K(x)\bigr) = - r^2 s^2 |x-p|^{-4}$. Moreover, for $|x-p|\geq s/2$, we have
\begin{equation} |\nabla_g \tilde{u}(x)|^2 = |\nabla_g u(K(x))|^2 \frac{s^2 r^2}{|x-p|^4} \label{grad}\end{equation}
Therefore,
\[ \hspace{-0.9cm} \log \int_{A_p(s,r)} e^u \,dV_g + \log \int_{A_p(s,r)} e^{-u} \,dV_g +4\log s \,\,\, = \]
\begin{eqnarray*}
      &  =   & \log \int_{A_p(s,r)} e^u s^2 \,dV_g + \log \int_{A_p(s,r)} e^{-u} s^2 \,dV_g \\
      & \leq & \log \int_{A_p(s,r)} e^u\, \frac{\,s^2}{\,r^2} \,dV_g + \log \int_{A_p(s,r)} e^{-u}\, \frac{\,s^2}{\,r^2} \,dV_g + C \\
      & \leq & \log \int_{A_p(s,r)} e^u\, \frac{s^2 r^2}{|x-p|^4} \,dV_g + \log \int_{A_p(s,r)} e^{-u}\, \frac{s^2 r^2}{|x-p|^4} \,dV_g + C, 
\end{eqnarray*}
where we have used the trivial inequality $r \geq |x-p|$ for $x\in A_p(s,r)$. By using (\ref{exp}), applying Proposition \ref{mt-local} to $\tilde{u}$ and then using (\ref{grad}), we have
\[ \log \int_{A_p(s,r)} e^u\, \frac{s^2 r^2}{|x-p|^4} \,dV_g + \log \int_{A_p(s,r)} e^{-u}\, \frac{s^2 r^2}{|x-p|^4} \,dV_g + C \,\,\, = \]
\begin{eqnarray*}
   &   =  & \log \int_{A_p(s,r)} e^{u(K(x))} \,dV_g + \log \int_{A_p(s,r)} e^{-u(K(x))} \,dV_g + C \\
   & \leq & \frac{1}{16\pi-\varepsilon} \int_{B_p(2r)}|\nabla_g \tilde{u}|^2 \,dV_g + C = \frac{1}{16\pi-\varepsilon} \int_{A_p(s/2,2r)}|\nabla_g \tilde{u}|^2 \,dV_g + C \\
   &  =   & \frac{1}{16\pi-\varepsilon} \int_{A_p(s/2,2r)}|\nabla_g u(K(x))|^2 \frac{r^2 s^2}{\,\,|x-p|^4} \,dV_g + C \\
   &  =   & \frac{1}{16\pi-\varepsilon} \int_{A_p(s/2,2r)}|\nabla_g u|^2 \,dV_g + C.
\end{eqnarray*}
This concludes the proof of the lemma.
\end{proof}

\begin{rem}
We are now able to prove the improved inequality given in Proposition \ref{mt-improved}. The spirit of the proof is to use jointly Lemmas \ref{ball} and \ref{mt-annulus}. Indeed, assume that $e^u$ and $e^{-u}$ concentrate around the same point at the same rate (in the sense of Proposition \ref{conc}). If we sum the inequalities given by Lemmas \ref{ball} and \ref{mt-annulus}, the extra term $4\log s$ cancels and we can deduce the improved inequality of Proposition \ref{mt-improved}.

We have to manage the case that when $\psi \left( \frac{e^{u}}{\int_{\Sigma} e^{u} dV_g} \right)= \psi \left( \frac{e^{-u}}{\int_{\Sigma} e^{-u} dV_g} \right)$ we do not really have concentration around the same point. Moreover, the property in Lemma \ref{mt-annulus} of $u$ being constant on the boundary of a ball need not be satified.
\end{rem}

\smallskip

\begin{proof}[Proof of Proposition \ref{mt-improved}]
Fixed $\varepsilon>0$, take $R>1$ (depending only on $\varepsilon$) and let $\psi$ be the continuous map given by Proposition \ref{conc}. Fix also $\delta>0$ small. 

Let $u\in H^1(\Sigma)$ be a function with $\int_\Sigma u\,dV_g=0$, such that
\[
\psi \left( \frac{e^u}{\int_{\Sigma} e^u \,dV_g} \right)= \psi \left( \frac{e^{-u}}{\int_{\Sigma} e^{-u} \,dV_g} \right)= (\beta, \sigma) \in \overline{\Sigma}_{\delta}. 
\]
If $\sigma \geq \frac{\delta}{R^2}$, then applying Proposition \ref{mt-imp} we get the result. Therefore, assume $\sigma < \frac{\delta}{R^2}$. Proposition \ref{conc} implies the existence of $\tau>0$, $p_1,\ p_2 \in \Sigma$ satisfying:
\begin{equation}  \label{dentro}
\int_{B_{p_1}(\sigma)} e^u \,dV_g \geq \tau \int_\Sigma
e^u \,dV_g, \qquad\hspace{0.1cm} \int_{B_{p_2}(\sigma)} e^{-u} \,dV_g \geq \tau \int_\Sigma
e^{-u} \,dV_g
\end{equation}
and
\begin{equation} \label{fuori}
\int_{B_{p_1}(R \sigma)^c} e^u \,dV_g \geq \tau \int_\Sigma e^u \,dV_g \qquad \int_{B_{p_2}(R \sigma)^c} e^{-u} \,dV_g \geq \tau \int_\Sigma e^{-u} \,dV_g ,
\end{equation}
with $d(p_1, p_2) \leq (6R+2) \sigma$. We divide the proof into two cases:

\medskip

\noindent {\bf CASE 1:} Assume that
\begin{equation} \label{caso1} 
\int_{A_{p_1}(R\sigma, \delta)} e^{u} \,dV_g \geq \tau/2 \int_{\Sigma} e^{u}
\,dV_g, \qquad \int_{A_{p_2}(R\sigma, \delta)} e^{-u} \,dV_g \geq \tau/2 \int_{\Sigma} e^{-u}
\,dV_g.
\end{equation}
In order to satisfy the hypothesis of Lemma \ref{mt-annulus}, we need to modify our function outside a certain ball. Via a dyadic decomposition, choose $k \in \mathbb{N}$, $k \leq 2 \varepsilon^{-1}$, such that
\[ 
\int_{A_{p_1}(2^{k-1} \delta, 2^{k+1} \delta)} |\nabla u|^2 \, dV_g \leq \varepsilon \int_{\Sigma} |\nabla u|^2 \, dV_g.
\]
We define $\tilde{u} \in H^1(\Sigma)$ by:
\[ 
\left \{ \begin{array}{ll} \tilde{u}(x) = u(x) & x \in B_{p_1}(2^k \delta), \\ \Delta \tilde{u}(x) =0 & x \in A_{p_1}(2^k \delta, 2^{k+1} \delta), \\ \tilde{u}(x) = c & x \notin B_{p_1}(2^{k+1} \delta),
         \end{array} \right.  
\]
where $c\in\mathbb{R}$. Moreover, since we want to apply Lemma \ref{mt-annulus} to $\tilde{u}$, we have to choose $\delta$ small enough so that $2^{3\varepsilon^{-1}}\delta < r_0$, where $r_0$ is given by that lemma.
 
We have that 
\begin{equation}\label{modify}
\begin{array}{ccl}
\displaystyle \int_{A_{p_1}(2^{k-1} \delta, 2^{k+1} \delta)} |\nabla \tilde{u}|^2 \, dV_g & \leq & \displaystyle C\int_{A_{p_1}(2^{k-1} \delta, 2^{k+1} \delta)} |\nabla u|^2 \, dV_g \\ \\
\displaystyle & \leq  & \displaystyle C \varepsilon \int_{\Sigma} |\nabla u|^2 \, dV_g,
\end{array}
\end{equation}
for some universal constant $C>0$.

\medskip

\noindent {\bf Case 1.1:} Suppose that $d(p_1, p_2) \leq R^{\frac 12} \sigma$.

\medskip

We first apply Lemma \ref{ball} to $u$ for $p=p_1$ and $s= 2(R^{1/2}+1)\sigma$, and take into account \eqref{dentro}, to obtain:
\begin{eqnarray} 
&& \frac{1}{16\pi-\varepsilon} \int_{B_p(s)}|\nabla u|^2 \,dV_g \,\,\,\geq
\nonumber \\ \nonumber \\
&\geq &\log  \int_{B_p(s/2)} e^{u} \,dV_g +\log \int_{B_p(s/2)}
e^{-u} \,dV_g - 4 \log \sigma - C \nonumber \\ \nonumber \\
&\geq& \log  \int_{\Sigma} e^{u} \,dV_g +\log \int_{\Sigma}
e^{-u} \,dV_g  - 4 \log \sigma - C. \label{dentro1} 
\end{eqnarray}
We next apply Lemma \ref{mt-annulus} to $\tilde{u}$ for $p=p_1$, $s'=4(R^{1/2}+1)\sigma$ and $r = 2^{k+1}\delta$:
\begin{equation}\label{fuori1.1}
\begin{array}{c}
\displaystyle \frac{1}{16\pi-\varepsilon}\int_{A_p(s'/2,2 r)}|\nabla_g \tilde{u}|^2 \,dV_g \,\,\, \geq \\ \\ 
\displaystyle \geq \,\,\, \log \int_{A_p(s', r)} e^{\tilde{u}} \,dV_g + \log \int_{A_p(s',r)} e^{-\tilde{u}} \,dV_g + 4\log \sigma - C. 
\end{array}
\end{equation}
Using the estimate \eqref{fuori}, we get 
\begin{equation} \label{fuori1}
\begin{array}{c}
\displaystyle \frac{1}{16\pi-\varepsilon}\int_{A_p(s'/2,2 r)}|\nabla_g \tilde{u}|^2 \,dV_g \,\,\, \geq \\ \\ 
\displaystyle \geq \,\,\, \log \int_\Sigma e^{u} \,dV_g + \log \int_\Sigma e^{-u} \,dV_g + 4\log \sigma - C. 
\end{array}
\end{equation}
Finally, combining \eqref{dentro1}, \eqref{fuori1} and \eqref{modify} we obtain our thesis (after renaming $\varepsilon$ conveniently).

\medskip

\noindent {\bf Case 1.2:} Suppose $d(p_1, p_2) \geq R^{\frac 12} \sigma$
and 
\[ 
\int_{B_{p_1}(R^{1/3} \sigma)} e^{-u}\, dV_g
\geq \tau/4 \int_\Sigma e^{-u} \,dV_g.
\]

\medskip

Here we argue as in Case 1.1. First, we apply Lemma \ref{ball} to $u$ for $p=p_1$ and $s = 2 (R^{1/3}+1)\sigma$. Then we use Lemma \ref{mt-annulus} with $\tilde{u}$ for $p=p_1$, $s'= 4 (R^{1/3}+1)\sigma$ and $r=2^{k+1} \delta$.

\medskip

\noindent {\bf Case 1.3:} Suppose $d(p_1, p_2) \geq R^{\frac 12} \sigma$
and  
\[
\int_{B_{p_2}(R^{1/3} \sigma)} e^{u} \,dV_g
\geq \tau/4 \int_\Sigma e^{u} \,dV_g.
\]

\medskip

This case can be treated as in Case 1.2, just  by interchanging the
indices.

\medskip

\noindent {\bf Case 1.4:} Suppose $d(p_1, p_2) \geq R^{\frac 12} \sigma$ and
\[ 
\int_{B_{p_2}(R^{1/3} \sigma)} e^{u} \,dV_g \leq
\tau/4 \int_\Sigma e^{u} \,dV_g, \qquad
\int_{B_{p_1}(R^{1/3} \sigma)} e^{-u} \,dV_g \leq \tau/4 \int_\Sigma
e^{-u} \,dV_g.
\]

\medskip

Take $n \in \mathbb{N}$, $n \leq 2 \varepsilon^{-1}$ so that
\[ 
\sum_{i=1}^2 \int_{A_{p_i}(2^{n-1} \sigma, 2^{n+1} \sigma )} |\nabla u|^2 \, dV_g \leq \varepsilon \int_{\Sigma} |\nabla u|^2 \, dV_g, 
\] 
where we have chosen $R$ such that $2^{3\varepsilon^{-1}} <R^{1/3}$. We define now the function \mbox{$v\in H^1(\Sigma)$} by:
\[
\left \{ \begin{array}{ll} v(x) = u(x) & x \in B_{p_1}(2^{n} \sigma) \cup B_{p_2}(2^{n} \sigma), \\
\Delta v(x) =0 & x \in A_{p_1}(2^{n} \sigma, 2^{n+1} \sigma) \cup A_{p_2}(2^{n}\sigma, 2^{n+1} \sigma), \\
v(x) = 0 & x \notin B_{p_1}(2^{n+1} \sigma) \cup B_{p_2}(2^{n+1}
\sigma).
\end{array} \right.  
\]
As before we have that
\begin{eqnarray*}
\sum_{i=1}^2 \int_{A_{p_i}(2^{n} \sigma, 2^{n+1} \sigma)} |\nabla v|^2 \, dV_g & \leq & C \sum_{i=1}^2 \int_{A_{p_i}(2^{n-1}  \sigma, 2^{n+1} \sigma)} |\nabla u|^2 \, dV_g \nonumber\\ \nonumber\\
& \leq & C \varepsilon \int_{\Sigma} |\nabla u|^2 \, dV_g,  
\end{eqnarray*}
where $C>0$ is a universal constant. 

Taking into account \eqref{dentro}, we now apply Lemma \ref{ball} to $v$ with $p=p_1$ and \mbox{$s=4(6R+2)\sigma$}:
\[
\frac{1}{16\pi-\varepsilon}\int_{ B_{p_1}(2^{n} \sigma) \cup B_{p_2}(2^{n} \sigma)} |\nabla u|^2 \,dV_g + C \varepsilon \int_{\Sigma} |\nabla u|^2 \, dV_g \,\,\, \geq \] \[
\geq \,\,\, \frac{1}{16\pi-\varepsilon}\int_{B_p(s)}|\nabla v|^2 \,dV_g
\]
\begin{eqnarray}
& \geq & \log\int_{B_p(s/2)} e^{v} \,dV_g +\log \int_{B_p(s/2)} e^{-v}
\,dV_g - 4 \log \sigma - C  \nonumber\\ \nonumber\\ 
& \geq &\log\int_{\Sigma} e^{u} \,dV_g + \log \int_{\Sigma} e^{-u} \,dV_g - 4 \log \sigma -C. \label{dentro2}
\end{eqnarray} 
Next, we define $w \in H^1(\Sigma)$ by:
\[
\left \{ \begin{array}{ll} w(x) = 0 & x \in B_{p_1}(2^{n} \sigma) \cup B_{p_2}(2^{n} \sigma), \\
\Delta w(x) =0 & x \in A_{p_1}(2^{n} \sigma, 2^{n+1} \sigma) \cup A_{p_2}(2^{n}\sigma, 2^{n+1} \sigma), \\
w(x) = \tilde{u}(x) & x \notin B_{p_1}(2^{n+1} \sigma) \cup B_{p_2}(2^{n+1}
\sigma).
\end{array} \right.  
\]
Again we have
\begin{eqnarray*}
\sum_{i=1}^2 \int_{A_{p_i}(2^{n} \sigma, 2^{n+1} \sigma)} |\nabla w|^2 \,dV_g & \leq &  C \sum_{i=1}^2 \int_{A_{p_i}(2^{n-1}  \sigma, 2^{n+1} \sigma)} |\nabla u|^2 \, dV_g \\ \\
& \leq & C \varepsilon \int_{\Sigma} |\nabla u|^2 \, dV_g, 
\end{eqnarray*}
where also here $C$ is a universal constant.

We apply Lemma \ref{mt-annulus} to $w$ for any point $p'$ such that $d(p', p_1) = \frac 1 2 R^{1/3}\sigma$, $s'= \sigma$ and $r = 2^{k+1} \delta$, to obtain:
\[
\frac{1}{16\pi-\varepsilon}\int_{ ( B_{p_1}(2^{n+1} \sigma) \cup B_{p_2}(2^{n+1} \sigma))^c} |\nabla u|^2 \, dV_g  + C \varepsilon \int_{\Sigma} |\nabla u|^2 \, dV_g \,\,\, \geq \] \[ \geq  \,\,\, \frac{1}{16\pi-\varepsilon}\int_{A_{p'}(s'/2,2 r)} |\nabla w|^2 \,dV_g  \] \[ \geq  \,\,\, \log \int_{A_{p'}(s', r)} e^{w} \,dV_g + \log \int_{A_{p'}(s',r)} e^{-w} \,dV_g + 4 \log \sigma - C.
\]
We now use \eqref{caso1} and the hypothesis of Case 1.4 to conclude that
\[
\frac{1}{16\pi-\varepsilon}\int_{ ( B_{p_1}(2^n \sigma) \cup B_{p_2}(2^n \sigma))^c} |\nabla u|^2 \, dV_g + C \varepsilon \int_{\Sigma} |\nabla u|^2 \, dV_g \,\,\, \geq
\]
\begin{equation}\label{fuori2} 
\geq \,\,\, \log \int_{\Sigma} e^{u} \,dV_g + \log \int_{\Sigma} e^{-u} \,dV_g + 4 \log \sigma - C.
\end{equation}
The inequality \eqref{fuori2} jointly with \eqref{dentro2} implies our result (after properly renaming $\varepsilon$).

\bigskip

\noindent {\bf CASE 2:} Assume that 
\[
\int_{B_{p_1}(\delta)^c} e^{u} \,dV_g \geq \tau/2 \int_{\Sigma} e^{u} \,dV_g \quad \mbox{or} \quad \int_{B_{p_2}(\delta)^c} e^{-u} \,dV_g \geq \tau/2 \int_{\Sigma} e^{-u} \,dV_g.
\]
Without loss of generality, suppose that the first alternative holds true. Let now \mbox{$\delta'= \frac{\delta}{2^{3/\varepsilon}}$}. If moreover:
\[
\int_{B_{p_2}(\delta')^c} e^{-u} \,dV_g \geq \tau/2 \int_{\Sigma} e^{-u} \,dV_g,
\]
then we can apply Proposition \ref{mt-imp} to deduce the thesis. Therefore we can assume that
\begin{equation}\label{caso2}
\int_{A_{p_2}(R\sigma,\delta')} e^{-u} \,dV_g \geq \tau/2 \int_{\Sigma} e^{-u} \,dV_g.
\end{equation}
We can apply the whole procedure of Case 1 to $u$, just by replacing $\delta$ with $\delta'$. In fact, as in Case 1.1, we would get the inequalities \eqref{dentro1} and \eqref{fuori1.1}. However, in this case we have to manage the fact that we do not know whether holds
\[
\int_{A_p(s',r)} e^{u} \,dV_g \geq \alpha \int_{\Sigma} e^{u} \,dV_g,
\]
for some fixed $\alpha>0$. This property is needed in \eqref{fuori1.1} to get the estimate 
\[
\log \int_{A_p(s', r)} e^{\tilde{u}} \,dV_g \geq \log \int_{\Sigma} e^{u} \,dV_g- C,
\]
which allows us to deduce \eqref{fuori1}. To do this, we first apply Jensen and Poincar\'e-Wirtinger inequalities, to get
\[
\log \int_{A_p(s', r)} e^{\tilde{u}} \,dV_g \geq \log \int_{A_p(r/8, r/4)} e^{u} \,dV_g \geq 
\]
\[
\log \fint_{A_{p_1}(r/8, r/4)} e^{u} \,dV_g - C \geq \fint_{A_{p_1}(r/8, r/4)} u \,dV_g - C \geq -\varepsilon \int_{\Sigma} |\nabla u|^2\, dV_g - C.
\]
Therefore, taking into account \eqref{caso2} and the last inequality, from \eqref{fuori1.1} we obtain (after properly renaming $\varepsilon$):
\begin{equation} \label{fuori3} 
\frac{1}{16\pi-\varepsilon}\int_{A_p(s'/2,2 r)}|\nabla\tilde{u}|^2 \,dV_g \geq \log \int_{\Sigma} e^{u} \,dV_g + 4 \log \sigma - C.
\end{equation}
Next, we apply Proposition \ref{mt-local}, to get
\[
\frac{1}{16\pi-\varepsilon} \int_{B_{p_1}(\delta/2)^c} |\nabla u|^2 \,dV_g \geq \log \int_{B_{p_1}(\delta)^c}  e^{u} \,dV_g + \log \int_{B_{p_1}(\delta)^c} e^{-u} \,dV_g .
\]
Reasoning as above and using the hypothesis of Case 2, we can deduce:
\begin{equation} \label{fuori4} 
\frac{1}{16\pi-\varepsilon}\int_{B_{p_1}(\delta)^c}|\nabla u|^2 \,dV_g \geq \log \int_{\Sigma} e^{u} \,dV_g + 4 \log \sigma - C.
\end{equation}
Finally we obtain our result by combining \eqref{fuori4}, \eqref{fuori3} and \eqref{dentro1}.

If we are under the conditions of Cases 1.2, 1.3 and 1.4, the thesis follows arguing in the same way.
\end{proof}

\begin{rem}\label{func-h}
Our goal is to use Proposition \ref{mt-improved} to obtain a lower bound of the functional $I_{\rho_1,\rho_2}$ under suitable conditions. The presence of the two functions $h_1$ and $h_2$ in $I_{\rho_1,\rho_2}$ is not so relevant because of the following estimates:
\begin{eqnarray*}
\log\int_\Sigma h_1(x)\,e^{u} \,dV_g &\leq& \log\int_\Sigma e^{u}\,dV_g +\log \|h_1\|_{\infty} \\\\
\log\int_\Sigma h_2(x)\,e^{-u} \,dV_g &\leq& \log\int_\Sigma e^{-u}\,dV_g +\log \|h_2\|_{\infty}
\end{eqnarray*}
\end{rem}

\section{Min-max scheme}

Let $\overline{\Sigma}_{\delta}$ be the topological cone over $\Sigma$ defined in \eqref{cone}, and let us set
\[ 
\overline{D}_{\delta} = diag\bigr(\overline{\Sigma}_{\delta} \times \overline{\Sigma}_{\delta}\bigr) = \bigr\{ (\vartheta_1, \vartheta_2)
    \in \overline{\Sigma}_{\delta} \times \overline{\Sigma}_{\delta} \; : \; \vartheta_1 = \vartheta_2 \bigr\}, 
\]
\[ 
X = \bigr(\overline{\Sigma}_{\delta} \times \overline{\Sigma}_{\delta}\bigr) \setminus \overline{D}_{\delta}. 
\]
Let $\varepsilon>0$ be sufficiently small and let $R, \delta, \psi$ be as in Proposition \ref{conc}. Consider then the map $\Psi$ defined by
\begin{equation}\label{eq:Psi}
    \Psi(u) = \left( \psi \left( \frac{e^{u}}{\int_{\Sigma} e^{u} \,dV_g}
    \right), \psi \left( \frac{e^{-u}}{\int_{\Sigma} e^{-u} \,dV_g} \right) \right).
\end{equation}
By Proposition \ref{mt-improved} and Remark \ref{func-h}, we have a lower bound of the functional $I_{\rho_1,\rho_2}$ on functions $u$ such that $u\in\overline{D}_{\delta}$. Therefore, there exists a large $L > 0$ such that if $I_{\rho_1, \rho_2}(u) \leq - L$ then it follows that $\Psi(u)\in X$.

\smallskip

In \cite{mal-ruiz} the authors proved that even though the set $X$ is non compact, it retracts to some compact subset $\mathcal{X}_\nu$. Indeed, we have the following lemma.

\begin{lem}\label{retr} 
For $\nu \ll \delta$, define
\[
  \mathcal{X}_{\nu,1} = \left\{ \bigr( (x_1, t_1), (x_2, t_2) \bigr) \in X :   \left| t_1 - t_2 \right|^2 + d(x_1, x_2)^2 \geq\delta^4,\right. \vspace{-0.25cm}\] 
\[ 
\hspace{0.2cm} \max\{t_1, t_2\} < \delta, \min\{t_1, t_2\} \in \left[ \nu^2, \nu \right] \Bigr\};
\]
$$
  \mathcal{X}_{\nu,2} = \Bigr\{ \bigr( (x_1, t_1), (x_2, t_2) \bigr) \in X
    \; : \; \max\{t_1, t_2\} = \delta, \min\{t_1, t_2\}
    \in \left[ \nu^2, \nu \right] \Bigr\},
$$
and set
\[
    \mathcal{X}_{\nu} = \bigr( \mathcal{X}_{\nu,1}
    \cup \mathcal{X}_{\nu,2} \bigr) \subseteq X.
\]
Then there is a retraction $R_{\nu}$ of $X$ onto $\mathcal{X}_{\nu}$.
\end{lem}

Our next goal is to introduce a family of test functions labelled on the set $\mathcal{X}_\nu$ on which the functional $I_{\rho_1, \rho_2}$ attains large negative values. For $(\vartheta_1, \vartheta_2) = \bigr( (x_1, t_1), (x_2, t_2) \bigr) \in \mathcal{X}_\nu$ define
\begin{equation}\label{test}
    \varphi(y) = \varphi_{(\vartheta_1, \vartheta_2)}(y) = \log  \frac{\left( 1 + \tilde{t}_2^2 d(x_2,y)^2\right)^2}{\left( 1 + \tilde{t}_1^2 d(x_1,y)^2 \right)^2},
\end{equation}
where
\[
    \tilde{t}_i = \tilde{t}_i(t_i) = \left\{
                    \begin{array}{ll}
                      \frac{1}{\,t_i} & \hbox{ for } t_i \leq \frac{\delta}{2},\vspace{0.1cm} \\ 
                      - \frac{4}{\delta^2} (t_i-\delta) & \hbox{ for } t_i \geq \frac{\delta}{2},
                    \end{array}
                  \right. 
\]
for $i=1,2$.

\smallskip

We start by proving the following estimate.

\begin{lem}\label{integral} 
For $\nu$ sufficiently small, and for $(\vartheta_1,\vartheta_2) \in \mathcal{X}_{\nu}$, there exists a constant $C=C(\delta,\Sigma) > 0$, depending only on $\Sigma$ and $\delta$, such that
\begin{equation}\label{estimate}
    \frac{1}{C} \frac{t_1^2}{t_2^4} \leq \int_\Sigma e^{\varphi}
    \,dV_g \leq C \frac{t_1^2}{t_2^4}.
\end{equation}
\end{lem} 

\begin{proof}
First, observe that the following equality holds true for some fixed positive constant $C_0$:
\begin{equation}\label{tot}
    \int_{\mathbb{R}^2} \frac{1}{\left( 1 + \lambda^2 |x|^2 \right)^2} \,dx = \frac{C_0}{\lambda^2};
     \qquad  \lambda > 0.
\end{equation}
To prove the lemma, we distinguish the two cases
\[
    |t_1 - t_2| \geq \delta^3 \qquad  \hbox{ and } \qquad |t_1 - t_2| < \delta^3,
\]
in order to exploit the properties of $\mathcal{X}_{\nu}$. Starting with the first alternative, by the definition of $\mathcal{X}_{\nu}$ and by the
fact that $\nu \ll \delta$, it turns out that one of the $t_i$'s belongs to $[\nu^2, \nu]$, while the other is greater or equal to $\frac{\delta^3}{2}$.

\smallskip

If $t_1 \in [\nu^2, \nu]$ and if $t_2 \geq \frac{\delta^3}{2}$ then the function $1 + \tilde{t}_2^2 d(x_2,y)^2$ is bounded above and below by two positive constants depending only on $\Sigma$ and $\delta$. Therefore, using \eqref{tot} we get
$$
    \frac{t_1^2}{C} = \frac{1}{C \tilde{t}_1^2} \leq \int_\Sigma
    e^{\varphi(y)} \,dV_g(y) \leq \frac{C}{\tilde{t}_1^2} = C t_1^2.
$$
On the other hand, if $t_2 \in [\nu^2, \nu]$ and if $t_1 \geq \frac{\delta^3}{2}$ then the function $1 + \tilde{t}_1^2 d(x_1,y)^2$ is bounded above and below by two positive constants depending only on $\Sigma$ and $\delta$, hence
$$
  \int_\Sigma e^{\varphi(y)} \,dV_g(y) \geq \frac{1}{C}
   \int_\Sigma \left(1 + \tilde{t}_2^2 d(x_2,y)^2\right)^2 \,dV_g(y) \geq
 \frac{\tilde{t}_2^4}{C} = \frac{1}{C t_2^4},
$$
and similarly
$$
  \int_\Sigma e^{\varphi(y)} \,dV_g(y) \leq C
   \int_\Sigma \left(1 + \tilde{t}_2^2 d(x_2,y)^2\right)^2 \,dV_g(y) \leq
 C \tilde{t}_2^4 = \frac{C}{t_2^4}.
$$
In both the last two cases we then obtain the conclusion.

\smallskip

Suppose now that we are in the second alternative, i.e. $|t_1 - t_2| < \delta^3$. Then by the definition of $\mathcal{X}_{\nu}$ we have that $d(x_1, x_2) \geq \frac{\delta^2}{2}$ and that $t_1, t_2 \leq \nu + \delta^3$. Using \eqref{tot} we obtain
$$
  \int_\Sigma e^{\varphi(y)} \,dV_g(y) \geq \int_{B_{x_1}(\delta^3)} e^{\varphi(y)} \,dV_g(y)
  \geq \frac{1}{C} \frac{\left(1 + \tilde{t}_2^2
 d(x_1,x_2)^2\right)^2}{\tilde{t}_1^2} \geq \frac{1}{C} \frac{t_1^2}{t_2^4}.
$$
In an analogous way we derive
$$
  \int_{B_{x_1}(\delta^3)} e^{\varphi(y)} \,dV_g(y) \leq C \frac{\left(1 +
  \tilde{t}_2^2 d(x_1,x_2)^2\right)^2}{\tilde{t}_1^2} \leq C \frac{t_1^2}{t_2^4}.
$$
Finally, by the estimate
$$
  \int_{(B_{x_1}(\delta^3))^c} e^{\varphi(y)} \,dV_g(y) \leq \frac{C}{\tilde{t}_1^4}
   \int_{(B_{x_1}(\delta^3))^c} \left(1 + \tilde{t}_2^2 d(x_2,y)^2\right)^2 \,dV_g(y)
   \leq C \frac{t_1^4}{t_2^4},
$$
we are done.
\end{proof}

\begin{rem}
Notice that for $e^{-\varphi}$ the same result holds true just by exchanging the indices of $t_1$ and $t_2$.
\end{rem}

\begin{prop}\label{low} 
For $(\vartheta_1, \vartheta_2) \in \mathcal{X}_{\nu}$, let $\varphi_{(\vartheta_1, \vartheta_2)}$ be defined as in \eqref{test}. Then
$$
  I_{\rho_1,\rho_2}(\varphi_{(\vartheta_1, \vartheta_2)}) \rightarrow - \infty \quad
  \hbox{ as } \nu \rightarrow 0,
$$
uniformly for $(\vartheta_1, \vartheta_2) \in \mathcal{X}_{\nu}$.
\end{prop}

\begin{proof}
We start by showing the following estimates:
\begin{equation}\label{est1}
    \int_{\Sigma} \varphi \,dV_g = 4 \bigr(1 + o_{{\delta}}(1)\bigr)
   \log t_1 - 4 \bigr(1 + o_{{\delta}}(1)\bigr) \log t_2;
\end{equation}
\begin{equation}\label{est2}
    \frac{1}{2}\int_{\Sigma} |\nabla_g \varphi|^2 \,dV_g \leq
    16 \pi \bigr(1 + o_{\delta}(1)\bigr) \log \frac{1}{t_1} +
    16 \pi \bigr(1 + o_{\delta}(1)\bigr) \log \frac{1}{t_2}.
\end{equation}
We begin by proving \eqref{est1}. It is convenient to divide $\Sigma$ into the two subsets
$$
   A_1 = B_{x_1}(\delta) \cup B_{x_2}(\delta); \qquad
   A_2 = \Sigma \setminus \mathcal{A}_1.
$$
Moreover, we write 
$$
  \varphi(y) = 2 \log  \left(1 + \tilde{t}_2^2 d(x_2,y)^2
  \right) - 2 \log \left( 1 + \tilde{t}_1^2 d(x_1,y)^2 \right).
$$
For $y \in A_2$ we clearly have that
$$
  \frac{1}{C_{\delta,\Sigma} t_1^2} \leq 1 + \tilde{t}_1^2 d(x_1,y)^2 \leq
  \frac{C_{\delta,\Sigma}}{t_1^2}; \qquad 
   \frac{1}{C_{\delta,\Sigma} t_2^2} \leq 1 + \tilde{t}_2^2 d(x_2,y)^2 \leq
  \frac{C_{\delta,\Sigma}}{t_2^2},
$$
therefore we derive
\[
   \int_{A_2} \varphi \,dV_g =
   4 \bigr(1 + o_{{\delta}}(1)\bigr) \log t_1 - 4 \bigr(1 + o_{{\delta}}(1)\bigr) \log t_2.
\]
Moreover, working in normal geodesic coordinates at $x_i$ one also finds
$$
   \int_{B_{\delta}(x_i)} \log \left( 1 + \tilde{t}_i^2 d(x_i,y)^2 \right)
   \,dV_g = o_\delta(1) \log t_i.
$$
Using jointly the last two inequalities we obtain \eqref{est1}.

\smallskip

We prove now \eqref{est2}. We have that
\begin{eqnarray*}
  \nabla_g \varphi(y) & = &
   2 \nabla_g \log \left( 1 + \tilde{t}_2^2
   d(x_2,y)^2 \right) - 2 \nabla_g \log \left( 1 + \tilde{t}_1^2 d(x_1,y)^2 \right)  \\ 
  & = & \frac{ 4\, \tilde{t}_2^2 d(x_2,y) \nabla_g d(x_2,y)}{1 + \tilde{t}_2^2 d(x_2,y)^2}
  - \frac{ 4\, \tilde{t}_1^2 d(x_1,y) \nabla_g d(x_1,y)}{1 + \tilde{t}_1^2 d(x_1,y)^2}.
\end{eqnarray*}
From now on we will assume, without loss of generality, that $t_1 \leq t_2$.
We distinguish between the case $t_2 \geq \delta^3$ and $t_2 \leq \delta^3$.

In the first case the function $1 + \tilde{t}_2^2 d(x_2,y)^2$ is uniformly
Lipschitz with bounds depending only on $\delta$, and therefore we have
$$
   \nabla_g \varphi(y) = - \frac{4 \tilde{t}_1^2 d(x_1,y) \nabla_g d(x_1,y)}{1 + \tilde{t}_1^2 d(x_1,y)^2}  + O_\delta(1).
$$
Let us fix a large constant $C_1 > 0$ and consider the subdivision of the surface $\Sigma$ into the three domains
\[
   B_1 = B_{x_1}(C_1 t_1); \qquad  B_2 = B_{x_2}(C_1 t_2);
  \qquad B_3 = \Sigma \setminus (B_1 \cup B_2).
\]
In $B_1$ we have that $|\nabla_g \varphi| \leq {C}{\tilde{t}_1}$, while
\begin{equation}\label{eq1}
    \frac{\tilde{t}_1^2 d(x_1,y) \nabla_g d(x_1,y)}{1 + \tilde{t}_1^2 d(x_1,y)^2} = \bigr(1 + o_{C_1}(1)\bigr) \frac{ \nabla_g d(x_1,y)}{d(x_1,y)} \qquad \quad \hbox{ in }
  \Sigma \setminus B_1.
\end{equation}
These estimates imply that
\begin{eqnarray*}
  \frac{1}{2}\int_{\Sigma} |\nabla_g \varphi|^2 \,dV_g & = &
  \int_{\Sigma \setminus B_1} |\nabla_g \varphi|^2 \,dV_g + o_{\delta}(1) \log \frac{1}{t_1} + O_\delta(1) \\ 
   & = & 16 \pi \int_{C_1 t_1}^1 \frac{dt}{t} + o_{\delta}(1) \log \frac{1}{t_1} + O_\delta(1) \\
    & = & 16 \pi \bigr(1 + o_{\delta}(1)\bigr) \log \frac{1}{t_1} + 16 \pi \bigr(1 + o_{\delta}(1)\bigr) \log \frac{1}{t_2} + O_\delta(1),
\end{eqnarray*}
recalling that $t_2 \geq \delta^3$.

If instead $t_2 \leq \delta_3$, by the definition of $\mathcal{X}_{\nu}$
we have that $d(x_1, x_2) \geq \frac{\delta^2}{2}$, and therefore $B_1 \cap B_2 = \emptyset$. Similarly to \eqref{eq1} we get
$$
  \left\{
    \begin{array}{ll}
     \displaystyle\frac{\tilde{t}_1^2 d(x_1,y) \nabla_g d(x_1,y)}{1 + \tilde{t}_1^2
     d(x_1,y)^2} = \bigr(1 + o_{C_1}(1)\bigr) \frac{ \nabla_g d(x_1,y)}{d(x_1,y)}  &  \vspace{0.2cm}\\
     \displaystyle\frac{\tilde{t}_2^2 d(x_2,y) \nabla_g d(x_2,y)}{1 + \tilde{t}_2^2
     d(x_2,y)^2} = \bigr(1 + o_{C_1}(1)\bigr) \frac{ \nabla_g d(x_2,y)}{d(x_2,y)}
      &
    \end{array}
  \right. \qquad  \hbox{ in } B_3.
$$
Moreover we have
$$
    \left | \nabla_g \varphi \right |\leq {C}{\tilde{t}_i} \quad \hbox{ in }
    B_i, \ i=1, 2.
$$
Therefore we find
\begin{eqnarray*}
  \frac{1}{2}\int_{\Sigma} |\nabla_g \varphi|^2 \,dV_g & = &
  \int_{B_3} |\nabla_g \varphi|^2 \,dV_g  + o_{\delta}(1) \log \frac{1}{t_1} + o_{\delta}(1) \log \frac{1}{t_2} + O_\delta(1) \\
  & = & 16 \pi \bigr(1 + o_{\delta}(1)\bigr) \log \frac{1}{t_1} + 16 \pi \bigr(1 + o_{\delta}(1)\bigr) \log \frac{1}{t_2} + O_\delta(1),
\end{eqnarray*}
for $t_2 \leq \delta^3.$ This concludes the proof of \eqref{est2}.

Finally, the estimates \eqref{est1} and \eqref{est1}, jointly with \eqref{estimate} and Remark \ref{func-h} yield the inequality
\[
    I_{\rho_1,\rho_2}(\varphi) \leq \bigr(2 \rho_1 - 16 \pi + o_\delta(1)\bigr)  \log t_1 + \bigr(2 \rho_2 - 16 \pi + o_\delta(1)\bigr) \log t_2 \rightarrow - \infty 
\]
as $\nu \rightarrow 0$, uniformly for $(\vartheta_1, \vartheta_2) \in \mathcal{X}_\nu$, since $\rho_1, \rho_2 > 8 \pi$.
\end{proof}

We next state a technical lemma, that will be of use later on.

\begin{lem}\label{scale1} 
Let $\varphi_{(\vartheta_1, \vartheta_2)}$ be as in \eqref{test}: then, for some $C=C(\delta,\Sigma)>0$, the following estimates hold uniformly in $(\vartheta_1, \vartheta_2) \in \mathcal{X}_\nu$:
\begin{equation}\label{eq-scale}
    \sup_{x \in \Sigma} \int_{B_x(r t_1)} e^{\varphi} \,dV_g
   \leq C r^2 \frac{t_1^2}{t_2^4} \qquad  \forall r >0.
\end{equation}
Moreover, given any $\varepsilon > 0$ there exists $C=C(\varepsilon, \delta, \Sigma)$, depending only on $\varepsilon$, $\delta$ and $\Sigma$ (but not on $\nu$), such that
\begin{equation}\label{eq-scale2}
  \int_{B_{x_1}(C t_1)} e^{\varphi} \,dV_g \geq (1 - \varepsilon)
  \int_{\Sigma} e^{\varphi} \,dV_g,
\end{equation}
uniformly in $(\vartheta_1, \vartheta_2) \in \mathcal{X}_\nu$.
\end{lem}

\begin{proof}
By the elementary inequalities $\left(1 + \tilde{t}_2^2 d(x_2,y)^2\right)^2 \leq \frac{C}{t_2^4}$ and $1 +\tilde{t}_1^2 d(x_1,y)^2 \geq 1$ we have
$$
  \int_{B_x(t_1 r)} e^{\varphi(y)} \,dV_g(y) \leq \frac{C}{t_2^4} \int_{B_x(t_1 r)}
  \frac{1}{\left( 1 + \tilde{t}_1^2 d(x_1,y)^2 \right)^2} \,dV_g(y) \leq C r^2
  \frac{t_1^2}{t_2^4} \qquad  \hbox{ for all } x \in \Sigma,
$$
which gives the inequality \eqref{eq-scale}.

We now prove \eqref{eq-scale2}. Using again that $\left(1 + \tilde{t}_2^2 d(x_2,y)^2\right)^2 \leq \frac{C}{t_2^4}$ we have that
\begin{equation}\label{equa}
    \int_{\Sigma \setminus B_{x_1}(R t_1)} e^{\varphi(y)} \,dV_g(y)
  \leq \frac{C}{t_2^4} \int_{\Sigma \setminus B_{x_1}(R t_1)}
  \frac{1}{\left( 1 + \tilde{t}_1^2 d(x_1,y)^2 \right)^2} \,dV_g(y).
\end{equation}
Finally, using normal geodesic coordinates centered at $x_1$ and \eqref{tot} with a change of variable, we find
$$
   \lim_{t_1 \rightarrow 0^+} t_1^{-2} \int_{\Sigma \setminus B_{x_1}(R t_1)}
  \frac{1}{\left( 1 + \tilde{t}_1^2 d(x_1,y)^2 \right)^2} \,dV_g = o_R(1)
  \qquad \hbox{ as } R \to + \infty.
$$
This fact and \eqref{equa}, with the estimate \eqref{estimate}, conclude the proof of the \eqref{eq-scale2}, by choosing $R$ sufficiently large, depending on $\varepsilon, \delta$ and $\Sigma$.
\end{proof}

\begin{rem}
The same result holds if we consider $e^{-\varphi}$, interchanging the indices of $t_1$ and $t_2$. 
\end{rem}

We next present a crucial step in describing the topology of low sublevels, which will allow us to find a min-max scheme later on.

\begin{prop}\label{hom} 
Let $L > 0$ be so large that $\Psi\bigr(\{ I_{\rho_1,\rho_2} \leq - L \}\bigr) \in X$,
and let $\nu$ be so small that $I_{\rho_1,\rho_2}(\varphi_{(\vartheta_1, \vartheta_2)}) < - L$ for $(\vartheta_1, \vartheta_2) \in \mathcal{X}_{\nu}$. Let $R_{\nu}$ be the retraction given in Lemma \ref{retr}. Then the map $T_\nu : \mathcal{X}_{\nu} \rightarrow \mathcal{X}_{\nu}$ defined as
$$
   T_\nu\bigr((\vartheta_1, \vartheta_2)\bigr) = R_{\nu} \bigr(\Psi(\varphi_{(\vartheta_1, \vartheta_2)})\bigr)
$$
is homotopic to the identity on $\mathcal{X}_{\nu}$.
\end{prop} 

\begin{proof}
Let us denote $\vartheta_i= (x_i, t_i)$ and
$$
f_1= \frac{e^{\varphi_{(\vartheta_1, \vartheta_2)}}}{\int_{\Sigma} e^{\varphi_{(\vartheta_1, \vartheta_2)}} \,dV_g}, \quad \psi
(f_1)=(\beta_1, \sigma_1),
$$ 
$$
f_2= \frac{e^{-\varphi_{(\vartheta_1, \vartheta_2)}}}{\int_{\Sigma} e^{-\varphi_{(\vartheta_1, \vartheta_2)}} \,dV_g}, \quad \psi
(f_2)=(\beta_2, \sigma_2),
$$ 
where $\psi$ is given in Proposition \ref{conc}. First, observe that we have the following relations
\begin{equation}\label{relation} 
\frac{1}{C} \leq \frac{\sigma_i}{t_i} \leq C, \qquad  d \left( \beta_i , x_i
    \right) \leq C t_i,
\end{equation}
for some constant $C=C(\delta,\Sigma)>0$, depending only on $\Sigma$ and $\delta$. Indeed, by \eqref{eq-scale2}, we have that
\[ 
\sigma\left(x_i, f_i \right) \leq C t_i,
\]
where $\sigma(x,f)$ is the continuous map defined in \eqref{def-sigma}. From that, we get that \mbox{$\sigma_i \leq C t_i$}. Moreover, by \eqref{eq-scale}, we get the relation $t_i \leq C \sigma_i$.

Next, by \eqref{dist} and using again the fact that $\sigma(x_i, f) \leq C t_i$, we obtain that
\[
d \bigr(x_i, S\left(f_i \right) \bigr )\leq C t_i,
\]
where $S(f)$ is the set defined in \eqref{defS}. But since we have the inequality
\[
d\bigr(\beta_i, S\left(f_i \right) \bigr) \leq C \sigma_i,
\]
we can conclude the proof of \eqref{relation}. 

We are now able to prove the proposition. The proof will follow by taking into account a composition of three homotopies. The first deformation $H_1$ is defined in the following way:
$$
  \Biggr( \left(
    \begin{array}{c}
      (\beta_1, \sigma_1) \\
      (\beta_2, \sigma_2) \\
    \end{array}
  \right), s \Biggr) \;\; 
  \stackrel{ H_1}{\longmapsto} \;\;  \left(
                                           \begin{array}{c}
                                           \bigr( \beta_1,\
                                     (1-s) \sigma_1 + s \kappa_1 \bigr) \\
                                   \\
                    
                      \bigr( \beta_2,\ (1-s) \sigma_2 + s \kappa_2 \bigr)
                                \end{array}
                                \right),
$$
where $\kappa_i= \min \left\{ \delta, \frac{\sigma_i}{\sqrt{\nu}} \right\}$.

We introduce now a second deformation $H_2$, given by
$$
  \Biggr( \left(
    \begin{array}{c}
      (\beta_1, \kappa_1) \\
      (\beta_2, \kappa_2) \\
    \end{array}
  \right), s \Biggr) \;\; 
  \stackrel{ H_2}{\longmapsto} \;\;  \left(
                                     \begin{array}{c}
                             \bigr( (1-s)\beta_1 + s x_1, \ \kappa_1 \bigr) \\
                                   \\
                                   
                             \bigr( (1-s)\beta_2 + s x_2,\ \kappa_2 \bigr)
                                \end{array}
                                \right),
$$
where $(1-s)\beta_i + s x_i$ stands for the geodesic joining $\beta_i$ and $x_i$ in unit time. Observe that, if $\kappa_i < \delta$, then we have that $\sigma_i < \sqrt{\nu} \delta$. Therefore by choosing $\nu$ small enough, we have that $\beta_i$ and $x_i$ are close to each other, by \eqref{relation}.
Instead, if $\kappa_i= \delta$, the equivalence relation in $\overline{\Sigma}_\delta$ makes the above deformation a trivial identification.

We perform a third deformation $H_3$ defined by
$$
  \Biggr( \left(
    \begin{array}{c}
      (x_1, \kappa_1) \\
      (x_2, \kappa_2) \\
    \end{array}
  \right), s \Biggr) \;\; 
  \stackrel{ H_3}{\longmapsto} \;\;  \left(
                                \begin{array}{c}
                                 \bigr( x_1,\ (1-s) \kappa_1 + s t_1 \bigr) \\
                                   \\
                                 \bigr( x_2,\ (1-s) \kappa_2 + s t_2 \bigr)
                                \end{array}
                               \right).
$$
Finally, we define $H$ as the composition of these three homotopies. Then,
$$ 
\bigr((\vartheta_1, \vartheta_2), s\bigr) \mapsto R_{\nu} \circ H\bigr(\Psi(\varphi_{(\vartheta_1, \vartheta_2)}),s\bigr)
$$ 
gives us the desired homotopy to the identity. Indeed, we observe that, since $\nu \ll \delta$, $H(\Psi(\varphi_{(\vartheta_1,\vartheta_2)}),s)$ always belongs to $X$, so that $R_{\nu}$ can be applied.
\end{proof}

\bigskip

We now introduce the min-max scheme which provides existence of solutions for equation \eqref{eq}. The argument follows the ideas of \cite{djlw}, which have been used extensively (see for instance \cite{djadli,dm,zhou}).

Let $\overline{\mathcal{X}}_{\nu}$ be the topological cone over $\mathcal{X}_{\nu}$, which can be represented as
\[
   \overline{\mathcal{X}}_{\nu} = \bigr(\mathcal{X}_{\nu} \times [0,1]\bigr) \Bigr/ \bigr(\mathcal{X}_{\nu} \times \{ 1 \}\bigr)
\] 
where the equivalence relation identifies all the points in $\mathcal{X}_{\nu} \times \{ 1 \}$. We choose $L > 0$ so large that $I_{\rho_1, \rho_2}(u) \leq - L$ implies that $\Psi(u)\in X$ and then $\nu$ so small that
\[
  I_{\rho_1,\rho_2}(\varphi_{(\vartheta_1, \vartheta_2)}) \leq - 4L
\]
uniformly for $(\vartheta_1, \vartheta_2) \in \mathcal{X}_{\nu}$. The existence of such $\nu$ is guaranteed by Proposition \ref{low}. Fixing this value of $\nu$, we define the following class:
\begin{equation}\label{class}
    \mathscr{H} = \Bigr\{ h : \overline{\mathcal{X}}_{\nu} \rightarrow H^1(\Sigma) \; : \; h \hbox{ is continuous and } h\bigr(\cdot\, \times\, \{0\}\bigr)
 = \varphi_{(\vartheta_1, \vartheta_2)} \hbox{ on } \mathcal{X}_{\nu} \Bigr\}.
\end{equation}
Then we have the following properties.

\begin{lem}\label{minmax-str}
The set $\mathscr{H}$ is non-empty and moreover, letting
\[
  c_{\rho_1,\rho_2} = \inf_{h \in \mathscr{H}}\; \sup_{m \in \overline{\mathcal{X}}_{\nu}} I_{\rho_1,\rho_2}\bigr(h(m)\bigr),
\]
one has that \,$c_{\rho_1,\rho_2} > - 2 L$.
\end{lem}

\begin{proof}
To prove that $\mathscr{H} \neq \emptyset$, we just notice that the map
\begin{equation} \label{test-map}
  \bar{h}(\vartheta,s) = s \,\varphi_{(\vartheta_1, \vartheta_2)},
  \qquad  (\vartheta,s) \in \overline{\mathcal{X}}_{\nu},
\end{equation}
belongs to $\mathscr{H}$. Assuming by contradiction that $c_{\rho_1,\rho_2}\leq -2L$ there would  exist a map $h \in \mathscr{H}$ with $\sup_{m \in \overline{\mathcal{X}}_{\nu}} I_{\rho_1,\rho_2}\bigr(h(m)\bigr) \leq - L$. Then, since Proposition \ref{hom} applies, writing $m = (\vartheta, t)$, with $\vartheta \in \mathcal{X}_{\nu}$, the map
\[
  t \mapsto R_{\nu} \circ \Psi \circ h(\cdot,t)
\]
would be a homotopy in $\mathcal{X}_{\nu}$ between $R_{\nu} \circ \Psi \circ \varphi_{(\vartheta_1, \vartheta_2)}$ and a constant map. But this is impossible since $\mathcal{X}_{\nu}$ is non-contractible (see the Remark \ref{non-contr} and by the fact that $\mathcal{X}_{\nu}$ is a retract of $X$) and since $R_{\nu} \circ \Psi \circ \varphi_{(\vartheta_1, \vartheta_2)}$ is homotopic to the identity on $\mathcal{X}_{\nu}$. Therefore we deduce the proof of the lemma.
\end{proof}

\begin{rem}\label{non-contr}
In \cite{mal-ruiz} the authors proved that the set $X = \overline{\Sigma}_{\delta} \times \overline{\Sigma}_{\delta} \setminus \overline{D}_{\delta}$ is non-contractible. Indeed, if $\Sigma=\mathbb{S}^2$, then $\overline{\Sigma}_{\delta}$ can be identified with $B_0(1) \subset \mathbb{R}^3$ and it turns out that $X\simeq \mathbb{S}^2$, where $\simeq$ stands for homotopical equivalence. The case of positive genus is not so easy. However, the authors proved that $X$ is non-contractible by showing that its cohomology group $H^4(X)$ is non trivial. 
\end{rem}

From the Lemma \ref{minmax-str}, the functional $I_{\rho_1,\rho_2}$ has a min-max structure. By classical arguments, such a structure yields a Palais-Smale sequence. However, we cannot directly conclude the existence of a critical point, since it is not known whether the Palais-Smale condition holds or not. To bypass this problem and get the conclusion, we need a different argument, usually taking the name `monotonicity argument'. This technique was first introduced by Struwe in \cite{struwe}, and than used in more general settings (see for example \cite{jean,djlw}).

\smallskip

Let us take $\mu > 0$ such that $\Lambda_i := [\rho_i-\mu, \rho_i+\mu]$ is contained in $(8 \pi, 16 \pi)$ for both $i = 1, 2$. We then consider $\tilde{\rho}_i \in \Lambda_i$ and the functional $I_{\tilde{\rho}_1,\tilde{\rho}_2}$ corresponding to these values of the parameters.

It is easy to check that the above min-max scheme applies uniformly for
$\tilde{\rho}_i \in \Lambda_i$ for $\nu$ sufficiently small. More precisely, given any large number $L > 0$, there exists $\nu$ so small that for $\tilde{\rho}_i \in \Lambda_i$ we have the gap:
\begin{equation}\label{eq:min-max}
   \sup_{m \in \partial
   \overline{\mathcal{X}}_{\nu}} I_{\tilde{\rho}_1,\tilde{\rho}_2}(m) < - 4 L; \qquad
    c_{\tilde{\rho}_1,\tilde{\rho}_2} := \inf_{h \in \mathscr{H}}
  \; \sup_{m \in \overline{\mathcal{X}}_{\nu}} I_{\tilde{\rho}_1,\tilde{\rho}_2}\bigr(h(m)\bigr) > -
  2L, 
\end{equation}
where $\mathscr{H}$ is defined in \eqref{class}. Moreover, using for
example the test map \eqref{test-map}, one shows that for $\mu$ sufficiently small there exists a large constant $\overline{L}$ such that for $\tilde{\rho}_i \in \Lambda_i$
\[
  c_{\tilde{\rho}_1,\tilde{\rho}_2} \leq \overline{L} 
\]
Under these conditions, the following proposition is well-known.

\begin{prop}\label{bounded-ps}
Let $\nu$ be so small that \eqref{eq:min-max} holds. Then the functional $I_{t \rho_1,t\rho_2}$ possesses a bounded Palais-Smale sequence $(u_n)_n$ at level $c_{\,t \rho_1,t\rho_2}$ for almost every $t \in \Gamma:=\left[ 1 - \frac{\mu}{16 \pi}, 1 + \frac{\mu}{16 \pi}\right]$.
\end{prop}

Using the above result we are now able to prove the Theorem \ref{main}.

\medskip

\begin{proof}[Proof of Theorem \ref{main}.]
The existence of a bounded Palais-Smale sequence for the functional $I_{t \rho_1,t\rho_2}$ implies by standard arguments that the functional possesses a critical point. Let now consider $t_j \rightarrow 1$, $t_j \in \Gamma$ and let $(u_j)_j$ denote the corresponding solutions. It is then sufficient to apply the compactness result in Theorem \ref{thm-comp}, which yields convergence of $(u_j)_j$ to a solution $u$ of \eqref{eq}, by the fact that $\rho_1, \rho_2$ are not multiples of $8 \pi$.
\end{proof}

\end{document}